\documentclass[12pt,reqno]{article}

\usepackage[usenames]{color}
\usepackage{amssymb}
\usepackage{amsmath}
\usepackage{amsthm}
\usepackage{amsfonts}
\usepackage{amscd}
\usepackage{graphicx}

\usepackage[colorlinks=true,
linkcolor=webgreen,
filecolor=webbrown,
citecolor=webgreen]{hyperref}

\definecolor{webgreen}{rgb}{0,.5,0}
\definecolor{webbrown}{rgb}{.6,0,0}

\usepackage{color}
\usepackage{fullpage}
\usepackage{float}

\usepackage{graphics}
\usepackage{latexsym}
\usepackage{epsf}
\usepackage{breakurl}

\newcommand{\seqnum}[1]{\href{https://oeis.org/#1}{\rm \underline{#1}}}

\DeclareMathOperator{\add}{add}
\DeclareMathOperator{\incr}{incr}

\DeclareMathOperator{\isk}{isk}
\DeclareMathOperator{\pair}{pair}
\def\andd{\, \wedge \,}
\def\Zee{\mathbb{Z}}

\begin{document}

\theoremstyle{plain}
\newtheorem{theorem}{Theorem}
\newtheorem{corollary}[theorem]{Corollary}
\newtheorem{lemma}[theorem]{Lemma}
\newtheorem{proposition}[theorem]{Proposition}

\theoremstyle{definition}
\newtheorem{definition}[theorem]{Definition}
\newtheorem{example}[theorem]{Example}
\newtheorem{conjecture}[theorem]{Conjecture}

\theoremstyle{remark}
\newtheorem{remark}[theorem]{Remark}

\title{Proof of Irvine's Conjecture via Mechanized Guessing}

\author{Jeffrey Shallit\footnote{Research supported by
NSERC grant 2018-04118.}\\
School of Computer Science\\
University of Waterloo\\
Waterloo, ON N2L 3G1 \\
Canada\\
\href{mailto:shallit@uwaterloo.ca}{\tt shallit@uwaterloo.ca}}

\maketitle

\begin{abstract}
We prove a recent conjecture of Sean A. Irvine about a nonlinear recurrence, using
mechanized guessing and verification.  Finite automata and
the theorem-prover {\tt Walnut} play a large role in the proof.\\

\noindent Key words:  Irvine's conjecture, Gutkovskiy's sequence, numeration system,
morphism, automaton, {\tt Walnut}, nonlinear recurrence, automatic sequence,
combinatorial game, subword complexity, critical exponent.
\end{abstract}

\section{Introduction}

Mathematicians have long used intelligent guessing of a problem's solution,
followed by rigorous verification (for example, by induction), to prove
theorems.  In this note I show how to do this, at least in some cases, using
a simple algorithm
to infer a finite automaton from empirical data.   Once a candidate automaton
is inferred, a rigorous proof of its correctness can be supplied by
using {\tt Walnut}, a theorem-prover for automatic sequences \cite{Mousavi:2016,Shallit:2022}.

On May 24 2017 Ilya Gutkovskiy proposed the following nonlinear recurrence as
sequence \seqnum{A286389} 
in the OEIS (On-Line Encyclopedia of Integer Sequences) \cite{Sloane:2023}:
\begin{equation}
g_n = \begin{cases}
	0, & \text{if $n=0$;} \\
	n - g_{\lfloor g_{n-1}/2 \rfloor}, & \text{otherwise.}
	\end{cases}
	\label{gut}
\end{equation}
The first few values of this sequence, which we call
Gutkovskiy's sequence, are given in Table~\ref{tab4}.
This recurrence is a variation on similar sequences originally discussed
by Hofstadter \cite[p.~137]{Hofstadter:1979}.

Then, on July 20 2022, Sean A. Irvine observed that this sequence
seemed to be given by the partial sums of the sequence
\seqnum{A285431}, which is the fixed point of the morphism
$h$, where $h(1) = 110$ and $h(0) = 11$.  We denote the sequence
\seqnum{A285431} by
$(k_n)_{n \geq 1}$, in honor of its proposer, Clark Kimberling.
The first few values
of the sequence \seqnum{A285431} are also given in Table~\ref{tab4};
in order to maintain the indexing given
in the OEIS, we define $k_0 = 0$.
More precisely, then, Irvine's conjecture is that
$g_n = \sum_{1 \leq i \leq n} k_i$.

\begin{table}[H]
\begin{center}
\begin{tabular}{c|ccccccccccccc}
$n$ & 0 & 1 & 2 & 3 & 4 & 5 & 6 & 7 & 8 & 9 & 10 & 11\\
\hline
$g_n$ & 0 & 1 & 2 & 2 & 3 & 4 & 4 & 5 & 6 & 7 & 8 & 8 \\
$k_n$ & 0 & 1 & 1 & 0 & 1 & 1 & 0 & 1 & 1 & 1 & 1 & 0 \\
\end{tabular}
\end{center}
\caption{First few values of $g_n$.}
\label{tab4}
\end{table}
In this note we prove Irvine's conjecture, as well as a number
of related results, using automata theory.

All the needed {\tt Walnut} code to verify the claims of the paper
is available on the author's website, 
\url{https://cs.uwaterloo.ca/~shallit/papers.html}.


\section{From a morphism to a numeration system}

We start with the morphism $h:1 \rightarrow 110$, $0 \rightarrow 11$
that generates OEIS sequence \seqnum{A285431}.
Define $K_n = h^n (1)$, so that $K_0 = 1$, $K_1 = 110$, $K_2 = 11011011$,
and so forth.  
\begin{proposition}
For $n \geq 2$ we have $K_n = K_{n-1} K_{n-1} K_{n-2} K_{n-2}$.
\label{prop1}
\end{proposition}
\begin{proof}
By induction on $n$.  The base cases of $n = 0,1$ are trivial.  Otherwise
assume $n\geq 2$.  Then
\begin{align*}
K_n &= h^n (1) = h^{n-1} (h(1)) = h^{n-1} (1) h^{n-1} (1) h^{n-1} (0) \\
&= K_{n-1} K_{n-1} h^{n-2} (11) =  K_{n-1} K_{n-1} K_{n-2} K_{n-2} .
\end{align*}
\end{proof}

Since each $K_i$ is the prefix of $K_{i+1}$, it follows that there is
a unique limiting infinite word ${\bf k} = k_1 k_2 k_3 \cdots =  1101101111 \cdots$ of
which all the $K_i$ are prefixes.   Furthermore, Proposition~\ref{prop1}
shows that $\bf k$ is a ``generalized automatic sequence'' as studied
in \cite{Shallit:1988}, and hence there is a numeration system associated
with it, where $k_n$ can be computed by a finite automaton taking, as inputs,
the representation of $n$ in this numeration system.

We now explain how this is done.
Define ${\cal K}_n = |K_n|$, so that
${\cal K}_0 = 1$, ${\cal K}_1 = 3$, ${\cal K}_2 = 8$, and in general
${\cal K}_n = 2{\cal K}_{n-1} + 2{\cal K}_{n-2}$.  This two-term linear
recurrence is sequence \seqnum{A028859} in
the OEIS (and also \seqnum{A155020} shifted by one).  
The Binet form for ${\cal K}_n$, which can be easily verified, is 
\begin{equation}
{\cal K}_n = \left({1\over 2} + {\sqrt{3}\over 3} \right) \gamma^n + 
\left({1\over 2}- {\sqrt{3}\over 3} \right) \delta^n,
\label{binet}
\end{equation}
where $\gamma = 1+\sqrt{3}$ and $\delta = 1-\sqrt{3}$.

We now build a numeration system, which we call {\it $K$-representation},
out of the sequence $({\cal K}_i)_{i \geq 0}$.
We represent every natural number as a sum
$\sum_{0 \leq i \leq t} a_i {\cal K}_i$,
where $a_i \in \Sigma_3 := \{ 0,1,2 \}$.  Furthermore we associate
a ternary word $a_t \cdots a_0$ 
with the corresponding sum, as follows:
\begin{equation}
[a_t \cdots a_0]_K := \sum_{0 \leq i \leq t} a_i {\cal K}_i.
\label{kimex}
\end{equation}
Notice that words are written ``backwards'' so the most significant digit is
at the left.

Evidently numbers could have multiple representations in this system
as we have described it so far.  For example $[22]_K = 8 = [100]_K$.
In order to get a unique, canonical representation, we impose
the restriction $a_i a_{i+1} \not= 22$.   This is in analogy with
a similar restriction for the Zeckendorf (or Fibonacci) numeration
system.   We let $(n)_K$ denote this canonical representation for $n$.
Table~\ref{tab1} gives the first few representations in this numeration system.
Notice that the canonical representation for $0$ is $\epsilon$, the
empty string.
\begin{table}[H]
\begin{center}
\begin{tabular}{c|c}
$n$ & $(n)_K$ \\
\hline
0 & $\epsilon$ \\
1 & 1 \\
2 & 2 \\
3 & 10 \\
4 & 11 \\
5 & 12 \\
6 & 20 \\
7 & 21 \\
8 & 100 \\
9 & 101 \\
10 & 102
\end{tabular}
\end{center}
\caption{Representation for the first few numbers.}
\label{tab1}
\end{table}

It is now easy to see that the greedy algorithm produces the
canonical representation \cite{Fraenkel:1985}.
Furthermore, it is easy to see that there is a finite
automaton that takes, as input, a string $x$ over the alphabet
$\Sigma_3$, and accepts if and only if $x$ is a canonical
representation.  It is depicted in Figure~\ref{fig1}.
(We routinely omit useless states without comment.)
\begin{figure}[htb]
\vskip -.5in
\begin{center}
\includegraphics[width=3in]{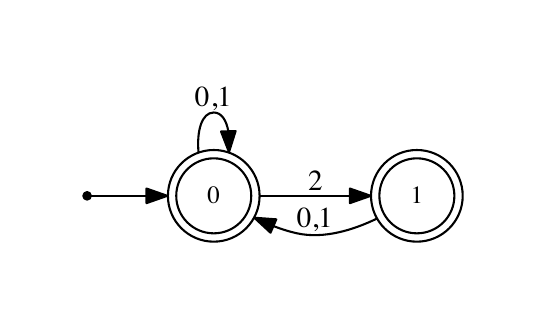}
\end{center}
\vskip -.5in
\caption{Automaton accepting canonical representations.}
\label{fig1}
\end{figure}

Some of the sequences we study in this paper were previously
studied by Fraenkel and co-authors
\cite{Fraenkel:1998,Artstein&Fraenkel&Sos:2008}, in the context
of some variations on Wythoff's game.  
These authors already found the numeration system we described here.
Also see \cite{Fokkink&Ortega&Rust:2022}.
Our main contribution
is to combine the use of automata theory with the numeration system.

\section{An incrementer automaton for $K$-representations}

We claim that we can go from the $K$-representation of $n$
to that of $n+1$ as follows:  if the last digit is $0$, add one to
it.  If the last digit is $1$, add one to it, except in the case
that the representation ends with $a(21)^i$, for $a \in \{ 0,1 \}$,
in which case the representation of $n+1$ ends in $(a+1) 0^{2i}$ instead.
If the last two digits are $a2$, for $a \in \{0,1\}$, then the
last two digits of $n+1$ are $(a+1) 0$.  Verification of this is
straightforward and is left to the reader.

A synchronized automaton `incr' implementing these rules is depicted in Figure~\ref{fig3}.  
The meaning of ``synchronized''
here is that the DFA takes the canonical $K$-representations of $n$ and $x$
in parallel as input, and accepts if $x = n+1$.
\begin{figure}[htb]
\begin{center}
\includegraphics[width=6in]{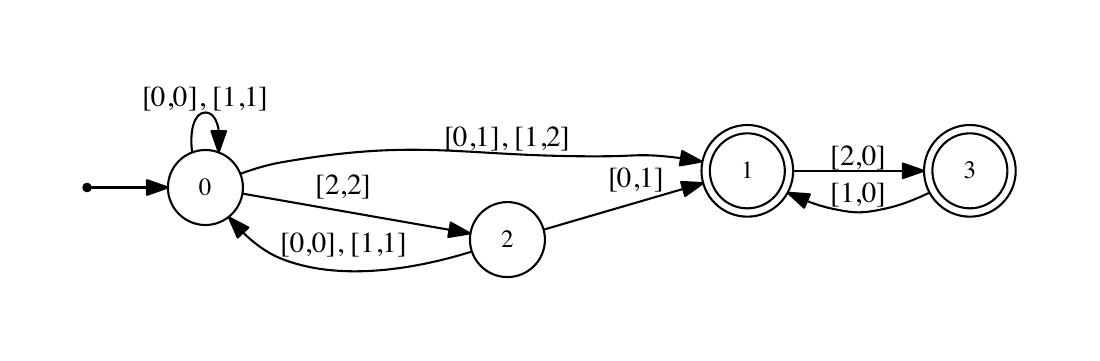}
\end{center}
\caption{Incrementer automaton for $K$-representations.}
\label{fig3}
\end{figure}

\section{An adder automaton for $K$-representation}
The next step is to build an ``adder'' for $K$-representations.
This is a synchronized automaton that takes, in parallel, the canonical
$K$-representations of integers $x,y, z$, and accepts
if and only if $x+y = z$.
The existence of this automaton for our numeration system
follows from very general results
of Frougny and Solomyak \cite{Frougny&Solomyak:1996}.

However, in this case it is actually easier to just ``guess''
the automaton from empirical data, and then verify its correctness.
The method of guessing is based on the Myhill-Nerode theorem
from formal language theory, and is explained, for example,
in \cite{Shallit:2022}.

Once we have an automaton that we believe is an adder, we can
verify its correctness by induction by checking the following conditions.
\begin{itemize}
\item[(i)] $\forall x, y\  \exists z \ \add(x,y,z)$   (adder is well-defined)
\item[(ii)] $\forall x,y,z,w\ (\add(x,y,z) \andd \add(x,y,w)) \implies z=w$ 
(adder represents a function)
\item[(iii)] $\forall x,y,z \ \add(x,y,z) \iff \add(y,x,z)$  (commutative law)
\item[(iv)] $\forall x,y,z,t \ (\exists r\ \add(x,y,r) \andd \add(r,z,t)) \iff
(\exists s\ \add(y,z,s) \andd \add(x,s,t)) $ (associative law)
\item[(v)] $\forall x \ \add(x,0,x)$  (base case of induction)
\item[(vi)] $\forall x,y \ \add(x,1,y) \iff \incr(x,y) $ (induction step).
\end{itemize}

Our candidate adder had $42$ states.  To verify its
correctness, we use the following
straightforward implementation of the conditions above.
\begin{verbatim}
eval check_i "?msd_kim Ax,y Ez $add(x,y,z)":
eval check_ii "?msd_kim Ax,y,z,w ($add(x,y,z) & $add(x,y,w)) => z=w":
eval check_iii "?msd_kim Ax,y,z $add(x,y,z) <=> $add(y,x,z)":
eval check_iv "?msd_kim Ax,y,z,t (Er $add(x,y,r) & $add(r,z,t)) <=> 
   (Es $add(y,z,s) & $add(x,s,t))":
eval checkv "?msd_kim Ax $add(x,0,x)":
eval checkvi "?msd_kim Ax,y $add(x,1,y) <=> $incr(x,y)":
\end{verbatim}
and {\tt Walnut} returns {\tt TRUE} for all six statements. The
correctness of the adder now follows.

We briefly comment on the syntax of {\tt Walnut} commands.  Here
{\tt A} and {\tt E} represent the universal and existential quantifiers
$\forall$ and $\exists$, respectively.   The jargon {\tt ?msd\_kim}
means to interpret the statements using the $K$-numeration system.
The symbol {\tt \&} means logical ``and'', {\tt |} means logical
``or'', {\tt \char'176} is logical negation, {\tt =>} is implication, and {\tt <=>} represents iff.
The command {\tt def} defines an automaton, {\tt eval} evaluates
truth or falsity, and {\tt reg} converts a regular expression to an
automaton.

\section{The Kimberling sequence}

Define $k'_n = k_{n+1}$ for $n \geq 0$.
It is now easy to create a DFAO (deterministic finite automaton
with output) computing the sequence
$(k'_n)_{n \geq 0}$, by associating states of the DFAO with letters
of the alphabet, and transitions with images of those letters, as
explained in \cite{Shallit:1988}.
It is depicted in Figure~\ref{fig2}.
\begin{figure}[htb]
\vskip -.5in
\begin{center}
\includegraphics[width=4in]{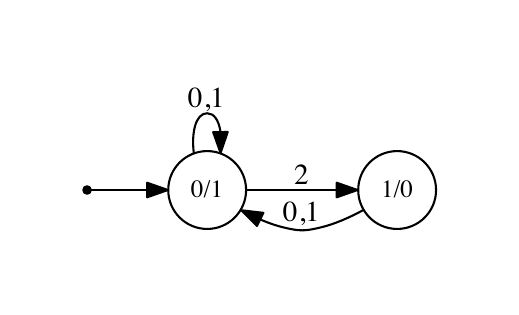}
\end{center}
\vskip -.5in
\caption{DFAO computing $k'_n$.}
\label{fig2}
\end{figure}
This DFAO takes a canonical $K$-representation of $n$ as
input, and outputs (as the last state reached) the value
of $k'_n$.  In {\tt Walnut} this is represented by
the file {\tt KP.txt}, as follows:

\vbox{
\begin{verbatim}
msd_kim

0 1
0 -> 0
1 -> 0
2 -> 1

1 0
0 -> 0
1 -> 0
\end{verbatim}
}

Once we have this DFAO, we can get a DFAO for $(k_n)_{n \geq 0}$
simply by shifting the index.
\begin{verbatim}
def kks "?msd_kim KP[n-1]=@1":
combine K kks:
\end{verbatim}
The resulting DFAO is depicted in Figure~\ref{fig7}.
\begin{figure}[htb]
\begin{center}
\includegraphics[width=5in]{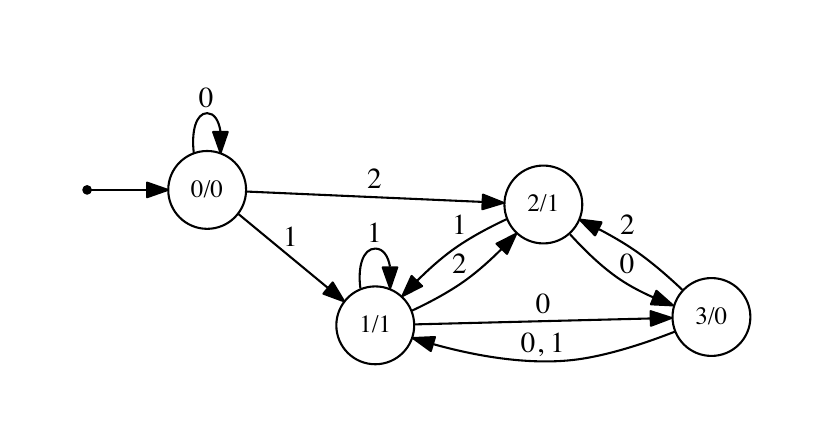}
\end{center}
\caption{DFAO for the sequence $\bf k$.}
\label{fig7}
\end{figure}

We can now verify that this automaton actually does compute
the Kimberling sequence.  We can do this by induction, by
verifying that
$${\bf k}[1..{\cal K}_n] = {\bf k}[1..{\cal K}_{n-1}] \ {\bf k}[1..{\cal K}_{n-1}] \ 
{\bf k}[1..{\cal K}_{n-2}]  \ {\bf k}[1..{\cal K}_{n-2}] .$$
To do so, we use the following {\tt Walnut} code:
\begin{verbatim}
reg isk msd_kim "0*10*":
reg pair msd_kim msd_kim "[0,0]*[1,0][0,1][0,0]*":
eval checkk1 "?msd_kim At,x ($isk(x) & t>=1 & t<=x) => K[t+x]=K[t]":
eval checkk2 "?msd_kim At,x,y ($pair(x,y) & t>=1 & t<=y) => K[t+2*x]=K[t]":
eval checkk3 "?msd_kim At,x,y ($pair(x,y) & t>=1 & t<=y) => K[t+2*x+y]=K[t]":
\end{verbatim}
Here $\isk(x)$ asserts that $x = {\cal K}_n$ for some $n \geq 1$, and
$\pair(x,y)$ asserts that $x = {\cal K}_{n+1}$ and $y = {\cal K}_n$ for some $n \geq 1$.

\section{Synchronized automaton for Gutkovskiy's sequence}

The last piece of the puzzle we need is a synchronized DFA 
computing Gutkovskiy's sequence \seqnum{A286389}.
To find this automaton we once again guess
it from empirical data, and
then verify it using Eq.~\eqref{gut}.  

The guessed $17$-state automaton is called `gut'.
To verify its correctness we use
the following {\tt Walnut} code:
\begin{verbatim}
eval check1 "?msd_kim An Ex $gut(n,x)":
eval check2 "?msd_kim An,x,y ($gut(n,x) & $gut(n,y)) => x=y":
eval check3 "?msd_kim $gut(0,0) & An,x,y,z (n>=1 & $gut(n,x) &
   $gut(n-1,y) & $gut(y/2,z)) => x+z=n":
\end{verbatim}
Thus our automaton correctly computes Gutkovskiy's sequence.

\section{Proof of Irvine's conjecture and more}

We now have everything we need to prove Irvine's conjecture.
\begin{theorem}
For $n \geq 0$ we have $g_n = \sum_{1 \leq i \leq n} k_i$.
\end{theorem}

\begin{proof}
We use the following {\tt Walnut} code:
\begin{verbatim}
eval check "?msd_kim An K[n]=@1 <=> (Ex $gut(n-1,x) & $gut(n,x+1))":
\end{verbatim}
and {\tt Walnut} returns {\tt TRUE}.
\end{proof}

Dekking, in the `formula' section of 
sequence \seqnum{A286389}, observed that
$g_n = (\sqrt{3}-1) n + O(1)$.  In fact we can prove a more exact
expression, a
kind of ``closed form'' for $g_n$.
\begin{theorem}
Define $\alpha = (\sqrt{3}-1)/2$ and $\beta = \sqrt{3}/3$.
We have
$$ g_n = \begin{cases}
	2 \lfloor \alpha n \rfloor + 1, & \text{if $[n]_K$ ends in $1$;} \\
	2 \lfloor \alpha n + \beta \rfloor, & \text{if $[n]_K$ ends in $0$ or $2$}.
	\end{cases}
$$
\end{theorem}

\begin{proof}
The starting point is the Binet form given in Eq.~\eqref{binet}.
From this, we easily verify that
\begin{equation}
{\cal K}_{i+1} - \gamma {\cal K}_i = (2-\sqrt{3}) \delta^i 
\label{binet2}
\end{equation}
for $i \geq 0$.

Now suppose $x = a_t a_{t-1} \cdots a_0 \in \{0,1,2\}^*$.
From \eqref{kimex} we have
$$ [x]_K = \sum_{0 \leq i \leq t} a_i {\cal K}_i$$
and
$$ [x0]_K = \sum_{0 \leq i \leq t} a_i {\cal K}_{i+1} .$$
Then, from \eqref{binet2}, we get
\begin{equation}
[x0]_K - \gamma [x0]_K = \sum_{0\leq i\leq t} a_i (2-\sqrt{3})\delta^i .
\label{diff}
\end{equation}
   
Since $-1 < \delta < 0$, we can bound the left-hand side
of \eqref{diff} by considering even powers of $\delta$ separately
from odd powers of $\delta$.  Summing to infinity, we get
\begin{equation}
-2 + {{2\sqrt{3}}\over 3} < [x0]_K - \gamma[x]_K < {{2\sqrt{3}}\over 3}.
\label{crucial1}
\end{equation}
This is one of the two crucial relations.

The second crucial relation, which can be proved by Walnut, is
\begin{equation}
g([xa]_K) = 2[x]_K + a.    
\label{crucial2}
\end{equation}
for $a \in \{0,1,2\}$.
Here I am writing $g()$ instead of $g_{}$ to make it easier to understand.
To prove it, we use the following {\tt Walnut} code:
\begin{verbatim}
reg has22 {0,1,2} "(0|1|2)*22(0|1|2)*":
reg lastd {0,1,2} {0,1,2} "()|([0,0]|[1,0]|[2,0])*([0,0]|[1,1]|[2,2])":
def lastdig "?msd_kim $lastd(n,x) & ~$has22(n)":
eval testeq "?msd_kim An,x,y,z ($gut(n,x) & $lastdig(n,y) & $kshift(n,z)) => 
	x=2*z+y":
\end{verbatim}
Here
\begin{itemize}
\item {\tt has22} checks for occurrence of the forbidden pattern $22$ in
an expansions;
\item {\tt lastd} takes two inputs $x$ and $y$ and accepts if $y$ is the
last digit of $x$;
\item {\tt lastdig} further enforces the condition that the inputs
be in the proper form for a Kimberling expansion; and
\item {\tt kshift} is a simple $3$-state automaton that accepts, in parallel,
inputs of the form $xa$ and $0x$.
\end{itemize}
Since the last command returns {\tt TRUE}, the result is proved.

Now let $n$ be a positive integer with Kimberling expansion $xa$, for
some string $x$ and $a\in \{0,1,2\}$.   
Then it is trivial that $n = [x0]_K + a$.
Multiply \eqref{crucial1} by $-2/\gamma$,
which reverses the inequalities, to get
\begin{equation}
{{2\sqrt{3}}\over 3} - 2
<  2[x]_K - (2/\gamma)[x0]_K < {{8 \sqrt{3}} \over 3} - 4 .
\label{ineq}
\end{equation}
Now add $a(1-2/\gamma)$ to both sides of \eqref{ineq} to get
\begin{equation}
{{2\sqrt{3}}\over 3} - 2 + a(1-2/\gamma) < 2[x]_K + a -(2/\gamma) ([x0]+a) <
{{8 \sqrt{3}} \over 3} - 4 + a(1-2/\gamma) .
\label{ineq2}
\end{equation}
Finally, since $n = [xa]$ and $g_n = 2[x]+ a$ and $[x0]+a = [xa]$ and
$1-2/\gamma = 2-\sqrt{3}$, we get
\begin{equation}
{{2\sqrt{3}}\over 3} - 2 + a(2-\sqrt{3}) < g_n -(2/\gamma)n   <
{{8 \sqrt{3}} \over 3} - 4 + a(2-\sqrt{3}) .
\label{ineq3}
\end{equation}

From \eqref{crucial2} we see that $g(n)$ is odd iff $a = 1$.  In this
case, setting $a=1$,
subtracting $1$ from \eqref{ineq3} and dividing by $2$, we get
$$
-0.7886751347\dots = -(\sqrt{3}/6 + 1/2) < (g(n)-1)/2 - {n\over\gamma} < 5 \sqrt{3}/6 - 3/2  = -0.0566243267\dots$$
and hence
$\lceil (g(n)-1)/2 - {n\over\gamma} \rceil = 0$.
But $(g(n)-1)/2$ is an integer, so we can shift it out of the ceiling expression
to get $(g(n)-1)/2 + \lceil -{n\over\gamma} \rceil = 0$.  Using $-\lfloor x \rfloor = \lceil -x \rceil$, we get 
$(g(n)-1)/2 - \lfloor {n\over\gamma} \rfloor = 0$
and hence $(g(n)-1)/2 = \lfloor  {n\over\gamma}\rfloor$.
Thus $g(n) = 2 \lfloor  {n\over\gamma}\rfloor + 1$.  

Now note that $g(n)$ is even iff either $a=0$ or $a = 2$.
Then, starting with \eqref{ineq3}, and dividing by $2$, we find
$$\sqrt{3}/3 - 1 < g(n)/2 -  {n\over\gamma}< \sqrt{3}/3.$$
Adding $1-\sqrt{3}/3$ to these inequalities gives
$$g(n)/2 -  {n\over\gamma}+ 1-\sqrt{3}/3 \in (0,1),$$
so
$ \lceil  g(n)/2 -  {n\over\gamma}+ 1-\sqrt{3}/3 \rceil = 1$.
But $g(n)/2 + 1$ is an integer, so we can pull it out of the ceiling to get
$g(n)/2 + 1 + \lceil -n/g - \sqrt{3}/3 \rceil = 1$.
Thus
$g(n)/2 + 1 - \lfloor n/g + \sqrt{3}/3 \rfloor = 1$, and
hence 
$g(n)/2 = \lfloor n/g + \sqrt{3}/3 \rfloor$, as desired.
\end{proof}

\begin{remark}
The idea of the proof follows the general lines of
a proof of Don Reble for Fibonacci
representations \cite{Reble:2008}.
\end{remark}


\section{Some related sequences and a problem of Fokkink, Ortega, and Rust}

We now turn to three related sequences; for $n\geq 1$
the first two give the $n$'th
positions of the ones (resp., zeros) in the sequence $\bf k$.
We call them $A_n$ and $B_n$, respectively.
The third sequence, called $Q_n$, has a more complicated definition:
\begin{equation}
Q_n = \begin{cases}
n, & \text{if $n \leq 1$;} \\
Q_m, & \text{if $n = Q_m + 2m$ and there is } \\
& \text{exactly one $i<n$ with $Q_i = Q_m$;}  \\
\text{least positive integer not in $Q_1, \ldots, Q_{n-1}$},
& \text{otherwise.}
\end{cases}
\end{equation}
It is sequence \seqnum{A026366} in the OEIS.
\begin{table}[H]
\begin{center}
\begin{tabular}{c|cccccccccccccccc}
$n$ & 0 & 1 & 2 & 3 & 4 & 5 & 6 & 7 & 8 & 9 & 10 & 11 & 12 & 13 & 14 & 15 \\
\hline
$A_n$ & 0 & 1 & 2 & 4 & 5 & 7 & 8 & 9 & 10 & 12 & 13 & 15 & 16 & 17 & 18 & 20 \\
$B_n$ & 0 & 3 & 6 & 11 & 14 & 19 & 22 & 25 & 28 & 33 & 36 & 41 & 44 & 47 & 50 & 55\\
$Q_n$ & 0 & 1 & 2 & 1 & 3 & 4 & 2 & 5 & 6 & 7 & 8 & 3 & 9 & 10 & 4 & 11 \\
\end{tabular}
\end{center}
\caption{First few values of $A_n$, $B_n$, and $Q_n$.}
\label{tab9}
\end{table}

Once again we can guess synchronized automata computing these
functions and verify that they are correct.  The guessed automaton
for $A_n$ has $23$ states, the guessed automaton for $B_n$ has
$24$ states, and the guessed automaton for $Q_n$ has 45 states.
We call them `aa', `bb', and `qq', respectively.

We now verify correctness of $A$ and $B$:
\begin{verbatim}
eval check_A_1 "?msd_kim An Ex $aa(n,x)":
eval check_A_2 "?msd_kim An,x,y ($aa(n,x) & $aa(n,y)) => x=y":
eval check_A_3 "?msd_kim Ax (En n>=1 & $aa(n,x)) <=> K[x]=@1":
eval check_A_4 "?msd_kim An,x,y ($aa(n,x) & $aa(n+1,y)) => x<y":
eval check_B_1 "?msd_kim An Ex $bb(n,x)":
eval check_B_2 "?msd_kim An,x,y ($bb(n,x) & $bb(n,y)) => x=y":
eval check_B_3 "?msd_kim Ax (En  $bb(n,x)) <=> K[x]=@0":
eval check_B_4 "?msd_kim An,x,y ($bb(n,x) & $bb(n+1,y)) => x<y":
\end{verbatim}
and {\tt Walnut} returns {\tt TRUE} for all of these.

To verify correctness of $Q$, we need to verify its definition:
\begin{verbatim}
def occurs_once_in "?msd_kim (Ei,x i>=1 & i<n & $qq(i,x) & $qq(m,x)) &
   (~Ei,j,x i>=1 & i<j & j<n & $qq(i,x) & $qq(j,x) & $qq(m,x))":
# true if Q_m occurs exactly once in Q_0, Q_1, ..., Q_{n-1} 

def occurs_in "?msd_kim Ei,y i<n & $qq(i,y) & $qq(i,x)":
# true if x occurs in Q_0, ..., Q_{n-1}

def least_not_in "?msd_kim (~$occurs_in(n,x)) &
   (Az (~$occurs_in(n,z)) => z>=x)":
# true if x is the least integer not in Q_1, ..., Q_{n-1}

eval check_Q_1 "?msd_kim An Ex $qq(n,x)":
eval check_Q_2 "?msd_kim An,x,y ($qq(n,x) & $qq(n,y)) => x=y":
eval check_Q_3 "?msd_kim Am,n,y,z (1<=m & m<n & $occurs_once_in(m,n) &
   $qq(m,y) & n=y+2*m & $qq(n,z)) => y=z":
eval check_Q_4 "?msd_kim An,y ($qq(n,y) & ~(Em 1<=m & m<n &
   $occurs_once_in(m,n))) => $least_not_in(n,y)":
\end{verbatim}
So indeed our automaton computes $Q_n$ correctly.

If we look at OEIS sequence \seqnum{A026367}, we see that its description
says (essentially) ``least $t$ such that $Q_t = n$".   
This allows use to verify that \seqnum{A026367} is in fact $A_n$, as follows:
\begin{verbatim}
def check_A_5 "?msd_kim An,t $aa(n,t) => $qq(t,n) & Au (u<t) => ~$qq(u,n)":
\end{verbatim}

Similarly, if we look at OEIS sequence \seqnum{A026368}, we see that its
description says (essentially) ``greatest $t$ such that $Q_t = n$".
We can then verify that \seqnum{A026368} is in fact $B_n$, as follows:
\begin{verbatim}
def check_B_5 "?msd_kim An,t $bb(n,t) => $qq(t,n) & Au (u>t) => ~$qq(u,n)":
\end{verbatim}
In particular, we have proved Neil Sloane's observation that ``\seqnum{A026368} 
appears to be [the] complement[ary] sequence of \seqnum{A026367}\,".

We can easily verify the observation of
Fokkink, Ortega, and Rust \cite{Fokkink&Ortega&Rust:2022} that
$B_n = 2A_n + n$ for $n \geq 0$:
\begin{verbatim}
eval check_FOR "?msd_kim An,x,y ($aa(n,x) & $bb(n,y)) => y=2*x+n":
\end{verbatim}
and {\tt Walnut} returns {\tt TRUE}.

Finally, Fokkink, Ortega, and Rust \cite{Fokkink&Ortega&Rust:2022}
left the following as an open problem, which we can turn into a theorem.
\begin{theorem}
For all $n$ we have $A_{B_n} \in \{ A_n + B_n - 1, A_n + B_n \}$.
\label{fokk}
\end{theorem}
\begin{proof}
We use the following {\tt Walnut} code:
\begin{verbatim}
eval check_FOR_2 "?msd_kim An,t,x,y ($aa(n,t) & $bb(n,x) & $aa(x,y)) => 
   (y=t+x|y+1=t+x)":
\end{verbatim}
and {\tt Walnut} returns {\tt TRUE}.
\end{proof}

\begin{remark}
Furthermore we could, if it were desired, give a DFAO
that computes, for each input $n$, which of the two
alternatives in Theorem~\ref{fokk} holds.

Similarly we can prove, for example, that
$B_{A_n} - A_n - B_n \in \{-3,-2,-1,0,1\}$.
\end{remark}

\section{Two more related sequences}

In this section we consider two additional related sequences:
$g'_n := g_n \bmod 2$, and $h_n := \sum_{0 \leq i < n} g'_n$.   The
first few terms are given in Table~\ref{tab6}.
\begin{table}[H]
\begin{center}
\begin{tabular}{c|ccccccccccccccccccccc}
$n$ &0& 1& 2& 3& 4& 5& 6& 7& 8& 9&10&11&12&13&14&15&16&17&18&19&20\\
\hline
$g'_n$ &0& 1& 0& 0& 1& 0& 0& 1& 0& 1& 0& 0& 1& 0& 0& 1& 0& 1& 0& 0& 1 \\
$h_n$ & 0& 0& 1& 1& 1& 2& 2& 2& 3& 3& 4& 4& 4& 5& 5& 5& 6& 6& 7& 7& 7 
\end{tabular}
\end{center}
\caption{First few values of $g'_n$ and $h_n$.}
\label{tab6}
\end{table}

\begin{theorem}
The sequence $(g'_n)_{n \geq 0}$ is sequence \seqnum{A284772} in the
OEIS, that is, it is the fixed point of the morphism
$u:  0 \rightarrow 01$, $1 \rightarrow 0010$.
\end{theorem}

\begin{proof}
First, we create an automaton (in the Kimberling numeration system) for
$g'_n$ with {\tt Walnut}:
\begin{verbatim}
def gp "?msd_kim Ex,y $gut(n,x) & x=2*y+1":
combine GP gp:
\end{verbatim}
which produces the automaton {\tt GP} computing $g'_n$ displayed
in Figure~\ref{fig11}.
\begin{figure}[H]
\vskip -.5in
\begin{center}
\includegraphics[width=6in]{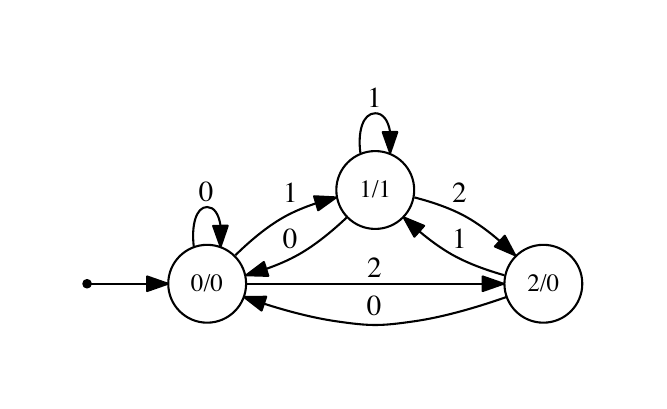}
\end{center}
\vskip -.5in
\caption{DFAO computing $g'(n)$.}
\label{fig11}
\end{figure}
From the transition diagram of this automaton, we can easily read off
the morphism $r: 0 \rightarrow 012$, $1 \rightarrow 012$, $2 \rightarrow 01$
and coding $s: 0,2\rightarrow 0$, $1 \rightarrow 1$, so that
$(g'_n)_{n \geq 0} = s(r^\omega(0))$.

It now remains to verify that $u^\omega(0) = s(r^\omega(0))$.  
To do this, we prove by induction on $n$ that
\begin{equation}
u^n (0) = s(r^{n-1} (01)) \quad \text{and} \quad  u^n(1) = s(r^{n-1} (2012).
\label{stu}
\end{equation}
The base case is $n = 1$ and is trivial.   Now assume
$n \geq 2$.  For the induction step, assume that \eqref{stu} holds
for $n' < n$.  Then
\begin{align*}
u^n(0) &= u^{n-1} (01) = s(r^{n-2} (01) r^{n-2} (2012)) 
= s(r^{n-2} (012012)) = s(r^{n-1} (01)) \\
u^n(1) &= u^{n-1} (0010) = s(r^{n-2} (01) r^{n-2} (01) r^{n-2} (2012) r^{n-2} (01)) \\
& = s(r^{n-2} (0101201201)) = s(r^{n-1} (2012)) ,
\end{align*}
as desired.
\end{proof}

\begin{theorem}
For $n \geq 0$ we have $g_n = 2 h_n + g'_n$.
\end{theorem}

\begin{proof}
We just sketch the proof, as the idea is similar to what we have
done before.  First, we ``guess'' a synchronized automaton
computing $(h_n)_{n \geq 0}$.   Then we verify it is correct
using the fact that we must have $h_{n+1} = h_n + g'_n$.  Finally,
we verify the equation $g_n = 2h_n + g'_n$.
\end{proof}

\section{Subword complexity}

Recall that the subword complexity function $\rho(n)$ counts the
number of distinct factors of length $n$ of an infinite word.  In this
section we compute this function for $\bf k$.

Call a factor $w$ of an infinite binary word $\bf x$ 
{\it right-special\/} if both $w0$ and
$w1$ appear in $\bf x$.
For binary words we know that $\rho(n+1)-\rho(n)$ counts the number of
length-$n$ right-special factors.   

{\tt Walnut} formulas for special factors are given in
\cite[\S 8.8.6]{Shallit:2022}.  Adapting them to our situation, we have
the following code:
\begin{verbatim}
def keqfac "?msd_kim At (t<n) => K[i+t]=K[j+t]":
def kisrs "?msd_kim Ej $keqfac(i,j,n) & K[i+n]!=K[j+n]":
eval nothree "?msd_kim Ei,j,k,n $kisrs(i,n) & $kisrs(j,n) & $kisrs(k,n) &
   ~$keqfac(i,j,n) & ~$keqfac(j,k,n) & ~$keqfac(i,k,n)":
def hastwo "?msd_kim Ei,j $kisrs(i,n) & $kisrs(j,n) & ~$keqfac(i,j,n)":
\end{verbatim}
Here
\begin{itemize}
\item {\tt keqfac} asserts that ${\bf k}[i..i+n-1] = {\bf k}[j..j+n-1]$;
\item {\tt kisrs} asserts that ${\bf k}[i..i+n-1]$ is a right-special factor;
\item {\tt nothree} asserts that there is no $n$ for which
$\bf k$ has three or more distinct right-special factors of length $n$;
\item {\tt hastwo} accepts precisely those $n$ for which $\bf k$
has exactly two distinct right-special factors of length $n$.
\end{itemize}

The automaton created by `hastwo' is displayed in Figure~\ref{fig8}.
\begin{figure}[htb]
\vskip -.5in
\begin{center}
\includegraphics[width=6in]{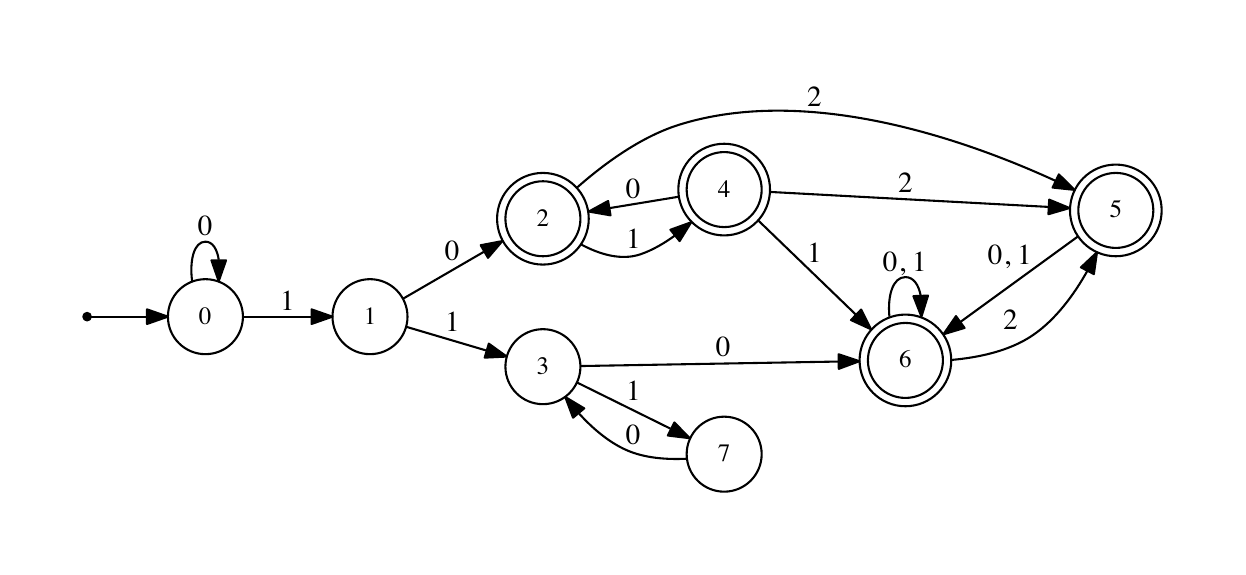}
\end{center}
\vskip -.5in
\caption{Automaton accepting those $n$ for which $\bf k$
has exactly two distinct right-special factors of length $n$.}
\label{fig8}
\end{figure}

We can now prove the following theorem.
\begin{theorem}
The infinite word 
$\bf k$ has exactly two distinct right-special factors of length $n$
if and only if there exists $i\geq 0$ such that one of the following
holds:
\begin{itemize}
\item $x \leq n < x+{\cal K}_{2i}$, where $x = 
{\cal K}_1 + {\cal K}_3 + \cdots + {\cal K}_{2i+1}$;
\item $y \leq n < y+{\cal K}_{2i+1}$, where
$y = {\cal K}_0 + {\cal K}_2 + \cdots + {\cal K}_{2i+2} $.
\end{itemize}
\end{theorem}

\begin{proof}
We use the following {\tt Walnut} code:
\begin{verbatim}
reg ul msd_kim msd_kim "[0,0]*[1,1][0,1]([1,1][0,0])*(()|[1,1]":
eval check_sw "?msd_kim An $hastwo(n) <=> Ex,y $ul(x,y) & x<=n & n<y":
\end{verbatim}
and {\tt Walnut} returns {\tt TRUE}.
\end{proof}

\begin{corollary}
We have $\limsup_{n \geq 1} \rho(n)/n = (30+\sqrt{3})/23 \doteq 1.37965438$
and $\liminf_{n \geq 1} \rho(n)/n = (3+\sqrt{3})/4 \doteq 1.1830127$.
\end{corollary}

\section{Critical exponents}

Recall that we say $p\geq 1$ is a {\it period\/} of a finite word $x = x[1..n]$
if $x[i]=x[i+p]$ for $1 \leq i \leq n-p$.   The {\it exponent\/} of a finite
word $x$ is the length of $x$ divided by its shortest period.
Finally, the {\it critical exponent\/}
of an infinite word $\bf z$ is the supremum,
over all finite nonempty factors $x$ of $\bf z$, of the exponent of $x$.
\begin{theorem}
The critical exponent of $\bf k$ is $(2\sqrt{3} + 12)/3 \doteq 5.1547$.
\end{theorem}

\begin{proof}
Since the basic ideas have already been covered elsewhere in detail 
\cite[pp.~148--150]{Shallit:2022}, we just
sketch them here.  We create {\tt Walnut} formulas for the shortest period
of a factor of $\bf k$, and then obtain the corresponding longest
words with the given period.  Then we restrict to those factors of
exponent at least $5$.  The resulting automaton, computed by 
`klong5', accepts pairs of the form
$(n,p) = ([121(01)^i0]_K, [10(00)^i0]_K)$ and
$(n,p) = ([121(01)^i02]_K, [10(00)^i00]_K)$.
Routine work with two-term linear recurrences then gives the result.

\begin{verbatim}
def kperi "?msd_kim p>0 & p<=n & Aj (j>=i & j+p<i+n) => K[j]=K[j+p]":
def klper "?msd_kim $kperi(i,n,p) & (Aq (q>=1 & q<p) => ~$kperi(i,n,q))":
def kleastp "?msd_kim Ei,n n>=1 & $klper(i,n,p)":
def klongest "?msd_kim (Ei $klper(i,n,p)) &
   (Ar,i $klper(i,r,p) => r<=n)":
def klong5 "?msd_kim $klongest(n,p) & n>5*p":
\end{verbatim}
\end{proof}

\section*{Acknowledgments}

I thank Michel Dekking for telling me about his paper
\cite{Dekking:2023}, and in particular its Remark 7.
I also acknowledge with thanks conversations with
Benoit Cloitre and Stefan Zorcic.

\end{document}